\documentclass[conference]{IEEEtran}
\IEEEoverridecommandlockouts
\usepackage{amsmath,amssymb,xcolor,fullpage,amsthm,graphicx}
\usepackage{cite}

\usepackage{tikz}

\allowdisplaybreaks
\allowbreak

\usepackage{algorithm}
\usepackage[noend]{algpseudocode}

\usepackage{mathdots}
\usepackage{yhmath}
\usepackage{cancel}
\usepackage{color}
\usepackage{siunitx}
\usepackage{array}
\usepackage{multirow}
\usepackage{amssymb}
\usepackage{gensymb}
\usepackage{tabularx}
\usepackage{booktabs}
\usetikzlibrary{fadings}
\usetikzlibrary{patterns}
\usetikzlibrary{shadows.blur}
\usetikzlibrary{shapes}

\newsavebox\mybox
\newenvironment{resizedtikzpicture}[1]{%
  \def\mywidth{#1}%
  \begin{lrbox}{\mybox}%
  \begin{tikzpicture}
}{%
  \end{tikzpicture}%
  \end{lrbox}%
  \resizebox{\mywidth}{!}{\usebox\mybox}%
}

\newcommand{\bi}{\begin{itemize}}
\newcommand{\ei}{\end{itemize}}
\newcommand{\vo}[1]{\textcolor[HTML]{000000}{\boldsymbol{#1}}}
\newcommand{\x}{\vo{x}}

\newcommand{\y}{\vo{y}}

\newcommand{\K}{\vo{K}}

\newcommand{\Q}{\vo{Q}}

\newcommand{\R}{\vo{R}}

\renewcommand{\H}{\vo{H}}
\newcommand{\I}{\vo{I}}

\newcommand{\A}{\vo{A}}

\newcommand{\n}{\vo{n}}

\newcommand{\w}{\vo{w}}
\newcommand{\W}{\vo{W}}

\newcommand{\F}{\vo{F}}

\newcommand{\real}{\mathbb{R}}
\newcommand{\Exp}[1]{\mathbb{E}\left[#1\right]}

\newcommand{\Var}[1]{\Exp{#1{#1}^T}}
\newcommand{\mup}{\vo{\mu}^{-}}

\newcommand{\Sig}{\vo{\Sigma}}
\newcommand{\mupp}{\vo{\mu}^{+}}

\newcommand{\trace}[1]{\mathbf{tr}\left(#1\right)}
\newtheorem{theorem}{Theorem}

\theoremstyle{remark}

\newcommand{\eqnlabel}[1]{\label{eqn:#1}}

   \newcounter{remarks}
\newenvironment{remarks}[1][]{\refstepcounter{remarks}\par\medskip
   \textbf{Remark~\theremarks. #1} \rmfamily}{\medskip}

\begin{document}

\title{Utility and Privacy in Object Tracking from Video Stream using Kalman Filter\\
}

\author{\IEEEauthorblockN{Niladri Das \& Raktim Bhattacharya}
\IEEEauthorblockA{\textit{Dept. of Aerospace Engineering} \\
\textit{Texas A\&M University}\\
College Station,TX, USA \\
niladridas@tamu.edu}
}

\maketitle

\begin{abstract}
Tracking objects in Computer Vision is a hard problem. Privacy and utility concerns adds an extra layer of complexity over this problem. In this work we consider the problem of maintaining privacy and utility while tracking an object in a video stream using Kalman filtering. Our first proposed method ensures that the localization accuracy of this object will not improve beyond a certain level. Our second method ensures that the localization accuracy of the same object will always remain under a certain threshold.
\end{abstract}


\section{Introduction}
We capture and share videos for a variety of purposes. These visual data has different private information \cite{Acquisti_2006}. The private information includes identity card, license plate number and finger-print. Another class of visual data, which is the focus of our paper, are the video streams of an object. We can use filtering algorithms (e.g. Kalman filter \cite{kalman}) to track with considerable precision. The object in motion is first detected by an image processing algorithm from the video frames. The accuracy depends on the algorithm, along with the resolution of the image frames. Higher resolution of the camera and higher accuracy of the detection algorithm in the \textit{pixel coordinate} improves localization of the tracked object in \textit{spatial coordinate}.

We address two important questions pertaining to tracking object using Kalman filter from a video stream. The first question is from a utility viewpoint. We define utility as the quality of the estimation accuracy. If we are putting together an image acquisition and detection system to track the object shown in Fig.~\ref{fig:redball} using Kalman filter  \cite{A__2010}, we can ask: what is the most economical setup that ensures the estimated localization error to be always below a prescribed threshold or with a utility greater than a prescribed threshold? 

The second question is about privacy. When an object is being tracked in a video stream, its privacy is proportional to the uncertainty in the estimate of its location. The notion of privacy is relevant when such videos are being accessed by a third party. The owner of this data might want to perturb the video such that a Kalman filter based estimation on it will keep the localization error above a prescribed value. Akin to the utility scenario one might ask: what is the optimal noise that we can add to the video which ensures that the estimated localization error is always above a prescribed threshold?

We are not aware of any prior works related to privacy and utility in object tracking using filtering from a video stream. Most of the works have focused on preserving privacy and/or utility of static images. In \cite{Orekondy_2018} the authors proposed a redaction by segmentation technique to ensure privacy of its contents. They showed that using their redaction method they can ensure near-perfect privacy while maintaining image utility. The authors in \cite{Boyle_2000} the authors studied the impact of filters that blur and pixelize at different levels on the privacy and utility of various elements in a video frame. In \cite{Winkler_2011} the authors presented a concept for user-centric privacy awareness in video surveillance. Other related works include \cite{Qureshi_2009},\cite{Brassil_2009},\cite{Saho_2018}, and \cite{Kim_2015}.

\section{Problem Formulation}
\begin{figure}
\centerline{\includegraphics[width=0.5\textwidth]{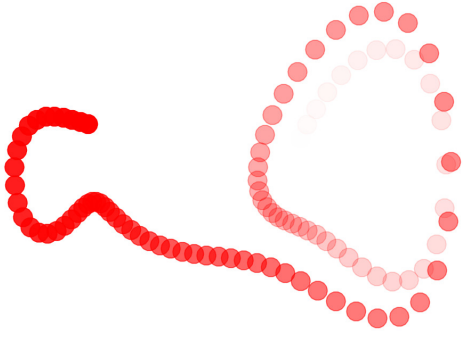}}
\caption{Time evolution of an object with darker shades representing more recent location.}
\label{fig:redball}
\end{figure}
We model the object detection process from a video frame using a linear discrete time stochastic systems $\bar{\mathcal{S}}$ described by the model of the form
\begin{subequations}\label{eq:LTI}
\begin{equation}
\x_{k+1} = \F\x_{k} + \w_{k},  \label{processDynamics}\\
\end{equation}
\begin{equation}
\y_k = \H\x_k + \n_k, \label{sensing}
\end{equation}
\end{subequations}
where $k=0,1,2,...$ represents the frame index, $\x_k\in\real^{n_{x}}$ is the $n_x$ dimensional true state of the \textit{model} in frame $k$, $\w_k \in\real^{n_x}$ is the  $n_x$ dimensional zero-mean Gaussian additive process noise variable with \sloppy $\mathbb{E}[\w_k\w_l^T] = \Q\delta_{kl}$. The $n_y$ dimensional observations in frame $k$ is denoted by $\y_k\in\real^{n_y}$ which is corrupted by an $n_y$ dimensional additive noise  $\n_k \in\real^{n_y}$. The sensor noise at each time instant is a zero mean Gaussian noise variable with $\mathbb{E}[\n_k\n_l^T] = \vo{R}\delta_{kl}$.
The initial conditions are $\Exp{\x_0}=\boldsymbol{\mu}_{0}$ and  $\Var{\x_0} = \Sig_0$. The process noise $\w_{k}$, observation noise $\n_k$, and initial state variable $\x_0$ are assumed to be independent.

The optimal state estimator for the stochastic system $\bar{\mathcal{S}}$ is the Kalman filter, defined by
\begin{align}
\K_k &= \Sig_k^-\H^T\Big[\H\Sig_k^-\H^T+\vo{R}\Big]^{-1} \tag*{(Kalman Gain)},\\
\mup_{k} & = \F\mupp_{k-1} \tag*{(Mean Propagation)},\\
\Sig_k^- &= \F\Sig^{+}_{k-1}\F^T+\Q \tag*{(Covariance Propagation)},\\
\mupp_{k}  &=\F\mupp_{k-1} + \K_k(\y_k-\H\mup_{k})\tag*{(Mean Update)},\\
\Sig^{+}_{k}&= (\I_{n_x}-\K_k\H)\Sig_k^-\tag*{(Covariance Update)},\\
\end{align}
where $\Sig_k^-,\Sig_k^+ \in \real^{n_x\times n_x}$ are the prior and posterior covariance matrix of the error estimate for frame $k$ respectively. The variables $\mup_{k},\mupp_{k}\in \real^{n_x}$, denote the prior and posterior mean estimate of the true state $\x_k$. The variable  $\boldsymbol{K}_k$ is the Kalman gain in frame $k$.  The parameter $\vo{R}$ is our design variable both for the case of utility and privacy, only varying in its interpretation. Now we define utility and privacy in the context of tracking a moving object.
\paragraph{Utility} Utility of the object detection system is specified by an upper bound on the steady-state estimation error due to filtering. We calculate a feasible $\vo{R}$ that ensures the steady state prior covariance matrix to be upper-bounded by a prescribed positive definite matrix $\vo{\Sigma}^d_{\infty}$ for the detection system modeled in eqn.~\ref{eq:LTI}. The parameter $\vo{R}$ is a measure of maximum inaccuracies allowed in the detection system.
\paragraph{Privacy} Privacy requirement is centered around a particular frame (say ${k+1}^{\text{th}}$).  It is specified by a lower bound on the estimation error $\vo{\Sigma}_{k+1}^{+}$ after the Kalman update, for that particular frame. This is where the privacy scenario differs from the utility case, where we focus on the steady-state error. We are interested in calculating a feasible $\vo{R}$ such that the posterior error covariance matrix  $\vo{\Sigma}_{k+1}^{+}$ is lower-bounded by a prescribed positive definite matrix $\vo{\Sigma}^{d}_{k+1}$. The parameter $\vo{R}$ is a measure of minimal noise that needs to be artificially injected to the ${k+1}^{\text{th}}$ image frame to ensure privacy with respect to accurate localization.

In the following sections we present two theorems that demonstrates how the utility and privacy preserving design parameter  $\vo{R}$ can be modeled as a solution to two convex optimization problems involving linear matrix inequalities (LMI).
\section{Optimal $\vo{R}$ for utility}
\begin{theorem}\label{thm:2}
Given $\vo{\Sigma}^d_\infty$, the desired steady-state error variance, the optimal algorithmic  precision $\vo{\Upsilon}:=\vo{R}^{-1}$ that satisfies $ \vo{\Sigma}_\infty \preceq \vo{\Sigma}^d_\infty$ is given by the following optimization problem,

\begin{equation}\left.
\begin{aligned}
& \min_{\vo{\Upsilon}}{ \trace{\vo{W}\vo{\Upsilon}\vo{W}^T}} \text{ subject to }\\
&\begin{bmatrix}
\vo{M}_{11} & \F\vo{\Sigma}^d_{\infty}\H^T \\
\H\vo{\Sigma}^d_{\infty}\F^T & \vo{L}+\vo{L}\vo{\Upsilon}\vo{L}
\end{bmatrix} \succeq 0,
\end{aligned}
\right\}
\label{eqn:thm2}
\end{equation}
where 
\begin{align*}
\vo{\Upsilon} &\succeq 0\\
\vo{L} &:=\H\vo{\Sigma}^d_{\infty}\H^T, \text{ and}\\
\vo{M}_{11} &:= \vo{\Sigma}^d_{\infty} - \F\vo{\Sigma}^d_{\infty}\F^T  - \Q \\
& + \F\vo{\Sigma}^d_{\infty}\H^T\vo{L}^{-1}\H\vo{\Sigma}^d_{\infty}\F^T, 
\end{align*}
with $\vo{\Upsilon} \in \real^{n_{y}\times n_{y}}$. The variable $\vo{W}\in\real^{n_{y}\times n_{y}}$, is user defined and serves as a normalizing weight on $\vo{\Upsilon}$.
\end{theorem}

\begin{proof}
 The steady state prior covariance is the solution to the following discrete-time algebraic Riccati equation (DARE)
\begin{multline}
\vo{\Sigma}_{\infty} = \F\vo{\Sigma}_{\infty}\F^T+ \Q \\ - \F\vo{\Sigma}_{\infty}\H^T\left(\H\vo{\Sigma}_{\infty}\H^T+{\vo{\Upsilon}^{-1}}\right)^{-1} \H\vo{\Sigma}_{\infty}\F^T, \label{ARE}
\end{multline}
where $\vo{\Upsilon}:=\vo{R}^{-1}$.
We assume that $\vo{\Sigma}^d_\infty$ is the solution of eqn.~\ref{ARE} for some $\vo{\Upsilon}^d \succeq 0$, i.e. for detection precision $\vo{\Upsilon}^d$ the steady-state variance is $\vo{\Sigma}^d_\infty$. 
We use $\vo{A}\succeq \vo{B}$ to denote that $\vo{A}-\vo{B}$ is a positive semi-definite matrix.
For any $\vo{\Upsilon}^d \preceq \vo{\Upsilon}$, eqn.~\ref{ARE} becomes the following inequality 
\begin{multline}
\vo{\Sigma}^d_{\infty} - \F\vo{\Sigma}^d_{\infty}\F^T  - \Q+  \\
 \F\vo{\Sigma}^d_{\infty}\H^T\left(\H\vo{\Sigma}^d_{\infty}\H^T+\vo{\Upsilon}^{-1}\right)^{-1} \H\vo{\Sigma}^d_{\infty}\F^T \succeq 0. 
 \eqnlabel{Ricc-relax}
 \end{multline}
Expanding $\left(\H\vo{\Sigma}^d_{\infty}\H^T+\vo{\Upsilon}^{-1}\right)^{-1}$ using matrix-inversion lemma, the above inequality becomes
\begin{multline}
\vo{\Sigma}^d_{\infty} - \F\vo{\Sigma}^d_{\infty}\F^T  - \Q 
 + \F\vo{\Sigma}^d_{\infty}\H^T\vo{L}^{-1}\H\vo{\Sigma}^d_{\infty}\F^T \\ - \F\vo{\Sigma}^d_{\infty}\H^T\left(\vo{L}+\vo{L}\vo{\Upsilon}\vo{L}\right)^{-1}\H\vo{\Sigma}^d_{\infty}\F^T \succeq 0, \eqnlabel{ricc-relax-final}
 \end{multline}
 where $\vo{L}:=\H\vo{\Sigma}^d_{\infty}\H^T$. Using Schur complement we get the following LMI
\begin{align}
\begin{bmatrix}
\vo{M}_{11} & \F\vo{\Sigma}^d_{\infty}\H^T \\
\H\vo{\Sigma}^d_{\infty}\F^T & \vo{L}+\vo{L}\vo{\Upsilon}\vo{L}
\end{bmatrix} \succeq 0, \eqnlabel{ss:LMI}
\end{align} 
where 
\begin{multline*}
\vo{M}_{11} := \vo{\Sigma}^d_{\infty} - \F\vo{\Sigma}^d_{\infty}\F^T  - \Q \\
 + \F\vo{\Sigma}^d_{\infty}\H^T\vo{L}^{-1}\H\vo{\Sigma}^d_{\infty}\F^T. 
\end{multline*}
The Optimal $\vo{\Upsilon}^\ast$ is achieved by minimizing $\trace{\vo{W}\vo{\Upsilon}\vo{W}^T}$.
\end{proof}
\begin{remarks}
We assume complete detectability of ($\F,\H$) and stabilizability of ($\F,\Q^{1/2}$) \cite{anderson1979optimal} for eqn.~\ref{eq:LTI}. This ensure existence and uniqueness of the steady state prior covariance matrix $\vo{\Sigma_{\infty}}$ (for a fixed $\R$) for the corresponding DARE in eqn.~\ref{ARE}. The linear matrix inequality (LMI) in eqn.~\ref{eqn:thm2} gives the feasible set of $\vo{R}:=\vo{\Upsilon}^{-1}$. We introduced the convex cost function $\trace{\vo{W}\vo{\Upsilon}\vo{W}^T}$ to calculate the most economical choice of $\vo{R}$.   
\end{remarks}
\subsection{Theoretical Bound on Utility}
The minimal steady-state covariance of the estimate that \textit{any object detection} setup can achieve modeled as in eqn.~\ref{eq:LTI}, is the solution to the following DARE
\begin{multline}
\vo{\Sigma}_{\infty} = \F\vo{\Sigma}_{\infty}\F^T+ \Q \\ - \F\vo{\Sigma}_{\infty}\H^T\left(\H\vo{\Sigma}_{\infty}\H^T\right)^{-1} \H\vo{\Sigma}_{\infty}\F^T .\label{eqn:noRDARE}
\end{multline}
This provides a theoretical lower bound on the prescribed $\vo{\Sigma}^d_{\infty}$ that we can achieve. A positive unique solution to $\vo{\Sigma}_{\infty}$ in eqn.~\ref{eqn:noRDARE} exists if $(\F,\H)$ pair is detectable, $(\F,\Q^{1/2})$ pair is stabilizable, and $\H\vo{\Sigma}_{\infty}\H^T$ is full-rank.

\section{Optimal $\vo{R}$ for privacy}
\begin{theorem}\label{thm:3}
Given $\vo{\Sigma}^d_{k+1}$, the desired predicted error variance at time $k+1$, the optimal measurement noise $\vo{R}_{\text{p}}$ that satisfies $ \vo{\Sigma}_{k+1}^{-} \succeq \vo{\Sigma}^d_{k+1}$ for a known $\vo{\Sigma}_{k}^{-}$, is given by the following optimization problem,

\begin{equation}\left.
\begin{aligned}
& \min_{\vo{R}_{\text{p}}}{ \trace{\vo{W}\vo{R}_{\text{p}}\vo{W}^T}} \text{ subject to }\\
&\begin{bmatrix}
\vo{M}_{11} & \vo{L} \\
\vo{L}^{T} & \vo{L}_2+\vo{R}_{\text{p}}
\end{bmatrix} \succeq 0,
\end{aligned}
\right\}
\label{eqn:thm3}
\end{equation}
where 
\begin{align*}
\vo{R}_{\text{p}} &\succeq 0\\
\vo{L}_1 &:=\F\vo{\Sigma}^{-}_{k}\H^T, \ \vo{L}_2 := \vo{H}\vo{\Sigma}^{-}_{k}\vo{H}^T+\vo{R}_{\text{s}}\text{ and}\\
\vo{M}_{11} &:= -\vo{\Sigma}^d_{k+1} + \F\vo{\Sigma}^{-}_{k}\F^T  +\Q,
\end{align*}
with $\vo{R}_{\text{p}} \in \real^{n_{y}\times n_{y}}$. The variable $\vo{W}\in\real^{n_{y}\times n_{y}}$, is user defined and serves as a normalizing weight on $\vo{R}_{\text{p}}$.
\end{theorem}

\begin{proof}

The Riccati equation for predicted covariance is
\begin{align*}
\vo{\Sigma}^{-}_{k+1} &= \A\vo{\Sigma}^{-}_{k}\A^T +\Q\\
&- \A\vo{\Sigma}^{-}_{k}\H^T(\H\vo{\Sigma}^{-}_{k}\H^T+\underbrace{\R_{s}+\R_{p}}_{\R})^{-1}\H\vo{\Sigma}^{-}_{k}\F^T
\end{align*}
where the measurement noise consists of inherent noise ($\R_{s}$) due to the object acquisition setup which is assumed to be known and the noise ($\R_{p}$) which needs to be added to ensure $ \vo{\Sigma}_{k+1}^{-} \succeq \vo{\Sigma}^d_{k+1}$. Here $\R_{p}$ is the design variable.

The $\R_{p}$ that ensures lower bound on $ \vo{\Sigma}_{k+1}^{-}$ satisfies
\begin{align*}
\vo{\Sigma}^{d}_{k+1} &\preceq \A\vo{\Sigma}^{-}_{k}\A^T +\Q\\
&- \A\vo{\Sigma}^{-}_{k}\H^T(\H\vo{\Sigma}^{-}_{k}\H^T+\R_{s}+\R_{p})^{-1}\H\vo{\Sigma}^{-}_{k}\F^T
\end{align*}
Using Schur complement we get the following linear matrix inequality,
\begin{align}
\begin{bmatrix}
\vo{M}_{11} & \vo{L} \\
\vo{L}^{T} & \vo{L}_2+\vo{R}_{\text{p}}
\end{bmatrix} \succeq 0, 
\end{align} 
where 
\begin{align*}
\vo{L}_1 &:=\F\vo{\Sigma}^{-}_{k}\H^T, \ \vo{L}_2 := \vo{H}\vo{\Sigma}^{-}_{k}\vo{H}^T+\vo{R}_{\text{s}}\text{ and}\\
\vo{M}_{11} &:= -\vo{\Sigma}^d_{k+1} + \F\vo{\Sigma}^{-}_{k}\F^T  + \Q,
\end{align*}
The optimal $\vo{R}^\ast_p$ is achieved by minimizing $\trace{\vo{W}\vo{R}_p\vo{W}^T}$.
\end{proof}
\begin{remarks}
The LMI in eqn.~\ref{eqn:thm3} gives the convex feasible set for $\vo{R}_p$ that ensures lower bound on the posterior covariance in the ${k+1}^{\text{th}}$ frame. We impose a cost convex cost function $\trace{\vo{W}\vo{R}_p\vo{W}^T}$ to calculate an optimal $\vo{R}_p$. 
\end{remarks}
\section{Numerical Results}
We assume a simplified motion model for the moving red object from one frame to another in a video, which is shown in Fig.~\ref{fig:redball}. The dynamics in the pixel frame is
\begin{align}
\underbrace{\begin{bmatrix}
x_{k+1} \\ y_{k+1} \\ \delta x_{k+1} \\ \delta y_{k+1}
\end{bmatrix}}_{\x_{k+1}^p} &=
\underbrace{\begin{bmatrix}
 1 & 0 & 1 & 0\\ 
 0 & 1 & 0 & 1\\
 0 & 0 & 1 & 0\\
 0 & 0 & 0 & 1
\end{bmatrix}}_{\F}
\underbrace{\begin{bmatrix}
x_{k} \\ y_{k} \\ \delta x_{k} \\ \delta y_{k}
\end{bmatrix}}_{\x_k^p}+\w_k,\label{eqn:pixeldyn}\\
\vo{y}_k &=
\underbrace{\begin{bmatrix}
1 & 0 & 0 & 0\\
0 & 1 & 0 & 0
\end{bmatrix}}_{\H}
\begin{bmatrix}
x_{k} \\ y_{k} \\ \delta x_{k} \\ \delta y_{k}
\end{bmatrix} + \n_k,\label{eqn:pixelmeas}
\end{align}
where $\x_{k}^p$ is the pixel coordinates of the moving object in the $k^{\text{th}}$ frame, $\mathbb{E}(\w_k,\w_l)=\delta_{kl}\vo{Q}$, and $\mathbb{E}(\n_k,\n_l)=\delta_{kl}\vo{R}$. The video is generated synthetically. There are a total of 500 frames in this video with 425 rows and 570 columns in each frame. The pair $(\F,\H)$ is completely detectable and $(\F,\Q^{1/2})$ is completely stabilizable, which ensures existence and uniqueness of positive solution to the induced DARE due to Kalman filtering of this system. The variable $\vo{R}$ is our design parameter. 

A homography exists between the pixel coordinates ($\x^p$) and the spatial coordinates ($\x$). The homography in this numerical problem is represented as an affine map
\begin{align*}
\x^{\text{p}} = 
\underbrace{\begin{bmatrix}
0 & \frac{n_r}{4}\\
-\frac{n_c}{4} & 0
\end{bmatrix}}_{\vo{U}}
\x + \begin{bmatrix}
\frac{n_r}{2} \\
\frac{n_c}{2}
\end{bmatrix}.
\end{align*}
The affine map induces a covariance relation $\vo{\Sigma}_{\x^{\text{p}}\x^{\text{p}}}= \vo{U}\vo{\Sigma}_{\x\x}\vo{U}^T$ from the pixel to the spatial coordinates.
\subsection{Utility results}
The optimal utility of an object detection setup, which includes the image acquisition hardware and the image processing algorithm,  can be prescribed as maximum error covariance allowed in the spatial coordinate frame ($\vo{\Sigma}_{\x\x}\preceq \vo{\Sigma}_{\x\x}^{\text{max}}$) due to filtering on the observed data. For instance, suppose we are tracking a car. We expect the tracking accuracy to be less than 
$\textbf{diag}([L_{\text{car}}^2 \ L_{\text{car}}^2])$, where $L_{\text{car}}$ denotes the length of the car. This is important from a situational awareness perspective in a traffic system. Using the induced covariance relation $\vo{\Sigma}_{\x^{\text{p}}\x^{\text{p}}}= \vo{U}\vo{\Sigma}_{\x\x}\vo{U}^T$, we transform the utility requirement into pixel coordinate system. The theoretical lower bound on utility in the pixel coordinate system for $\Q = \textbf{diag}([0.1 \ 0.1 \ 50 \ 50])$ is
\begin{align*}
\vo{\Sigma}_{\x^{\text{p}}\x^{\text{p}}}^{\text{lb}} =\textbf{diag}([54.891 \ 54.891]),   
\end{align*}
which can be solved using the \texttt{idare()} function in MATLAB \cite{MATLAB:2017}. This lower bound translates to a lower bound of 
\begin{align*}
\vo{\Sigma}_{\x\x}^{\text{lb}} =\textbf{diag}([2.693e-3 \ 4.845e-3]) \text{m}^2,   
\end{align*}
in the spatial coordinate system. If we allow for less precise filtering in pixel coordinates which can ensure a error covariance in the estimate of $1.5\vo{\Sigma}_{\x\x}^{\text{lb}}$, the convex optimization problem yields an optimal precision requirement of
$$\vo{\Upsilon}^* = \textbf{diag}([0.660 \ 0.660]),$$ with $\W$ chosen to be identity. To solve this we used CVX, a package for specifying and solving convex programs \cite{cvx},\cite{gb08}. We used SDPT3 solver \cite{T_t_nc__2003} which took a CPU time of 0.95 secs to solve the problem in CVX.

The calculated $\vo{R}^*:={\vo{\Upsilon}^*}^{-1}$ denotes that the intensity of the measurement noise (modeled as zero mean Gaussian) that gets added to the actual measurement due to the hardware and the object detection algorithm, needs to be less than $1.5$ (pixel length)$^2$. This will ensure that the estimation error always remains below the prescribed threshold of $1.5\vo{\Sigma}_{\x\x}^{\text{lb}}$. One can relate this precision requirement to different aspects of the detection process. For instance, the value of the precision is proportional to the resolution of the camera used. Higher resolution denotes higher precision. The matrix $\vo{W}$ used in the cost function can be interpreted as a price per unit resolution. With proper choice of $\vo{W}$ we can calculate the most economical sensing system that satisfies our requirement. Using a precision of $\vo{\Upsilon}^*= \textbf{diag}([0.660 \  0.660])$ we calculate the RMSE for 500 Monte-Carlo (MC) runs with randomized initial conditions which is shown in Fig.~\ref{fig:rmse1}. The peaks in the plot is due to the fact that we assumed a linear motion model whereas Fig.~\ref{fig:redball} shows that the motion no longer remains linear at places where there is considerable change in the direction. 
\begin{figure}
\centerline{\includegraphics[width=0.4\textwidth]{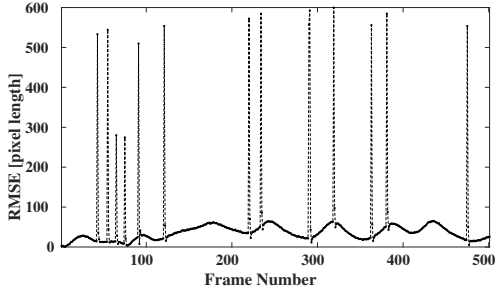}}
\caption{RMSE with 500 MC runs}
\label{fig:rmse1}
\end{figure}
\begin{figure}
\centerline{\includegraphics[width=0.5\textwidth]{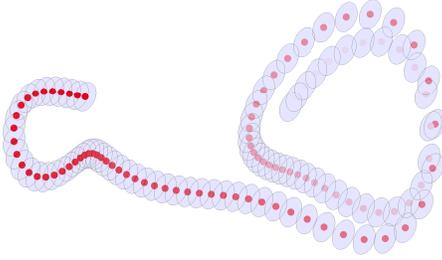}}
\caption{Error covariance averaged over 500 MC runs}
\label{fig:errorcov}
\end{figure}
In Fig.~\ref{fig:errorcov} we see the evolution of the error covariance matrix $\vo{\Sigma}_{\x\x}$ with different frames, averaged over 500 MC runs. 
In the steady state this covariance is guaranteed to remain below the prescribed $1.5\vo{\Sigma}_{\x\x}^{\text{lb}}$. 
\subsection{Privacy results}
In the system defined in eqn.~\ref{eqn:pixeldyn} and eqn.~\ref{eqn:pixelmeas} we assume that the measurement model has inherent sensor and/or object detection zero mean Gaussian noise ($\vo{n}_s$). We add a synthetic zero mean Gaussian noise ($\vo{n}_p$) to the image to ensure privacy. The noise intensity $\mathbb{E}[\vo{n}_s\vo{n}_s^T]=\vo{R}_s$ is known and $\mathbb{E}[\vo{n}_p\vo{n}_p^T]=\vo{R}_p$ is our design parameter.

\tikzset{every picture/.style={line width=0.75pt}} 
\begin{figure}
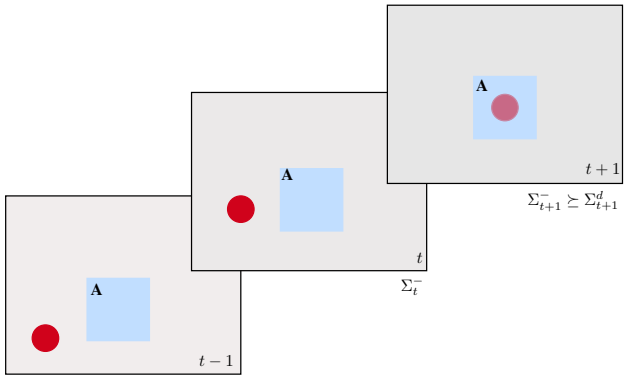

\begin{resizedtikzpicture}{0.5\textwidth}[x=0.75pt,y=0.75pt,yscale=-1,xscale=1] 
\draw  [fill={rgb, 255:red, 241; green, 237; blue, 237 }  ,fill opacity=1 ] (81,155) -- (269.5,155) -- (269.5,298) -- (81,298) -- cycle ;
\draw  [fill={rgb, 255:red, 235; green, 233; blue, 233 }  ,fill opacity=1 ] (230,72) -- (418.5,72) -- (418.5,215) -- (230,215) -- cycle ;
\draw  [fill={rgb, 255:red, 230; green, 230; blue, 230 }  ,fill opacity=1 ] (387,2) -- (575.5,2) -- (575.5,145) -- (387,145) -- cycle ;
\draw  [color={rgb, 255:red, 208; green, 2; blue, 27 }  ,draw opacity=1 ][fill={rgb, 255:red, 208; green, 2; blue, 27 }  ,fill opacity=1 ] (102.25,269.13) .. controls (102.25,263.26) and (107.01,258.5) .. (112.88,258.5) .. controls (118.74,258.5) and (123.5,263.26) .. (123.5,269.13) .. controls (123.5,274.99) and (118.74,279.75) .. (112.88,279.75) .. controls (107.01,279.75) and (102.25,274.99) .. (102.25,269.13) -- cycle ;
\draw  [color={rgb, 255:red, 208; green, 2; blue, 27 }  ,draw opacity=1 ][fill={rgb, 255:red, 208; green, 2; blue, 27 }  ,fill opacity=1 ] (258.88,165.63) .. controls (258.88,159.76) and (263.63,155) .. (269.5,155) .. controls (275.37,155) and (280.13,159.76) .. (280.13,165.63) .. controls (280.13,171.49) and (275.37,176.25) .. (269.5,176.25) .. controls (263.63,176.25) and (258.88,171.49) .. (258.88,165.63) -- cycle ;
\draw  [fill={rgb, 255:red, 208; green, 2; blue, 27 }  ,fill opacity=1 ] (470.63,84.13) .. controls (470.63,78.26) and (475.38,73.5) .. (481.25,73.5) .. controls (487.12,73.5) and (491.88,78.26) .. (491.88,84.13) .. controls (491.88,89.99) and (487.12,94.75) .. (481.25,94.75) .. controls (475.38,94.75) and (470.63,89.99) .. (470.63,84.13) -- cycle ;
\draw  [color={rgb, 255:red, 193; green, 222; blue, 255 }  ,draw opacity=1 ][fill={rgb, 255:red, 193; green, 222; blue, 255 }  ,fill opacity=1 ] (301.25,133.13) -- (351.25,133.13) -- (351.25,183.13) -- (301.25,183.13) -- cycle ;
\draw  [color={rgb, 255:red, 193; green, 222; blue, 255 }  ,draw opacity=1 ][fill={rgb, 255:red, 193; green, 222; blue, 255 }  ,fill opacity=1 ] (146.25,221.13) -- (196.25,221.13) -- (196.25,271.13) -- (146.25,271.13) -- cycle ;
\draw  [color={rgb, 255:red, 193; green, 222; blue, 255 }  ,draw opacity=1 ][fill={rgb, 255:red, 193; green, 222; blue, 255 }  ,fill opacity=1 ] (456.25,59.13) -- (506.25,59.13) -- (506.25,109.13) -- (456.25,109.13) -- cycle ;
\draw  [color={rgb, 255:red, 208; green, 2; blue, 27 }  ,draw opacity=0.4 ][fill={rgb, 255:red, 208; green, 2; blue, 27 }  ,fill opacity=0.53 ] (470.63,84.13) .. controls (470.63,78.26) and (475.38,73.5) .. (481.25,73.5) .. controls (487.12,73.5) and (491.88,78.26) .. (491.88,84.13) .. controls (491.88,89.99) and (487.12,94.75) .. (481.25,94.75) .. controls (475.38,94.75) and (470.63,89.99) .. (470.63,84.13) -- cycle ;

\draw (250,288) node    {$t-1$};
\draw (413,205) node    {$t$};
\draw (560,134) node    {$t+1$};
\draw (407,227) node    {$\Sigma ^{-}_{t}$};
\draw (536,159) node    {$\Sigma ^{-}_{t+1} \succeq \Sigma ^{d}_{t+1}$};
\draw (154,230) node   [align=left] {\textbf{A}};
\draw (307,138) node   [align=left] {\textbf{A}};
\draw (463,66) node   [align=left] {\textbf{A}};

\end{resizedtikzpicture}
\caption{Image frames with privacy in the region \textbf{A}}
\label{fig:threeframes}
\end{figure}
In Fig.~\ref{fig:threeframes} we see consecutive three frames with a smaller region inside them marked as \textbf{A}. These frames span the discrete time points $\{t,t-1,t+1\}$ as shown in the figure. When the tracked red object is in \textbf{A} in the $t+1^{\text{th}}$ frame, we want the location estimation error $\vo{\Sigma}^{-}_{t+1}$  to be greater than prescribed $\vo{\Sigma}^{d}_{t+1}$. We choose $\vo{\Sigma}^{d}_{t+1}$ to be $\textbf{diag}([2.703e-03 \ 4.862e-03])$m$^{2}$, in the spatial coordinates, which translates to $\textbf{diag}([54.891  \ 54.891])$ in the pixel frame.  Starting with an initial prior covariance $\vo{\Sigma}^{-}_{t}$, our proposed privacy theorem yields $$\vo{R}_p= \textbf{I}_2,$$ with $\vo{W}$ chosen to be identity. We assumed that the object acquisition and detection setup adds no noise the measurement, i.e. $\vo{R}_s = \textbf{0}$.
From a data sharing perspective, we would share  the image frame at time point $t+1$
 with added noise of intensity $\vo{R}_p$. 
Our privacy preserving framework is explained in Fig.~\ref{fig:frames}.
 
 To solve for $\vo{R}_p$ we again used CVX. We used SDPT3 solver which took a CPU time of 0.44 secs to solve the problem in CVX. The reduction in CPU time for the privacy problem compared to the utility problem is due to the fact that there is no inverse operation in the LMI. 
 
\begin{remarks}
We see in Fig.~\ref{fig:threeframes} that the red object which is being tracked using a Kalman filter, can still be identified in the $t+1$ frame, but cannot be precisely tracked beyond a certain accuracy.
\end{remarks} 

\tikzset{every picture/.style={line width=0.75pt}} 
\begin{figure}
\begin{resizedtikzpicture}{0.5\textwidth}[x=0.75pt,y=0.75pt,yscale=-1,xscale=1]

\draw  [fill={rgb, 255:red, 241; green, 237; blue, 237 }  ,fill opacity=1 ] (3,2) -- (94.5,2) -- (94.5,77) -- (3,77) -- cycle ;
\draw  [color={rgb, 255:red, 208; green, 2; blue, 27 }  ,draw opacity=1 ][fill={rgb, 255:red, 208; green, 2; blue, 27 }  ,fill opacity=1 ] (43.59,39.5) .. controls (43.59,36.42) and (45.9,33.93) .. (48.75,33.93) .. controls (51.6,33.93) and (53.91,36.42) .. (53.91,39.5) .. controls (53.91,42.58) and (51.6,45.07) .. (48.75,45.07) .. controls (45.9,45.07) and (43.59,42.58) .. (43.59,39.5) -- cycle ;
\draw  [fill={rgb, 255:red, 155; green, 155; blue, 155 }  ,fill opacity=1 ] (144,23) -- (376.5,23) -- (376.5,63) -- (144,63) -- cycle ;
\draw    (93.5,40) -- (141.5,40.96) ;
\draw [shift={(143.5,41)}, rotate = 181.15] [color={rgb, 255:red, 0; green, 0; blue, 0 }  ][line width=0.75]    (10.93,-3.29) .. controls (6.95,-1.4) and (3.31,-0.3) .. (0,0) .. controls (3.31,0.3) and (6.95,1.4) .. (10.93,3.29)   ;
\draw (292,229) node  {\includegraphics[width=52.5pt,height=52.5pt]{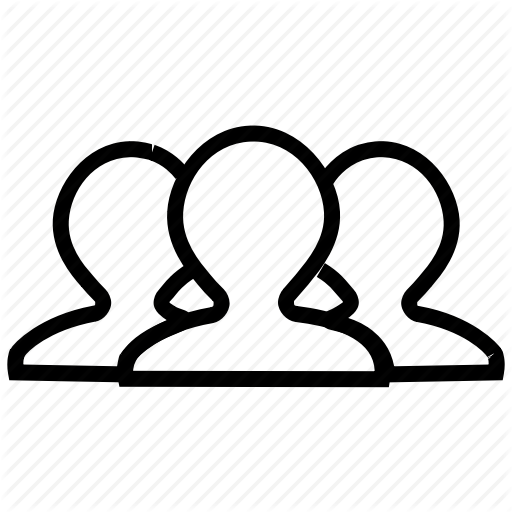}};
\draw  [fill={rgb, 255:red, 155; green, 155; blue, 155 }  ,fill opacity=1 ] (280,132) .. controls (280,124.27) and (286.27,118) .. (294,118) .. controls (301.73,118) and (308,124.27) .. (308,132) .. controls (308,139.73) and (301.73,146) .. (294,146) .. controls (286.27,146) and (280,139.73) .. (280,132) -- cycle ;
\draw    (294.5,64) -- (294.02,116) ;
\draw [shift={(294,118)}, rotate = 270.53] [color={rgb, 255:red, 0; green, 0; blue, 0 }  ][line width=0.75]    (10.93,-3.29) .. controls (6.95,-1.4) and (3.31,-0.3) .. (0,0) .. controls (3.31,0.3) and (6.95,1.4) .. (10.93,3.29)   ;
\draw    (294,146) -- (293.52,198) ;
\draw [shift={(293.5,200)}, rotate = 270.53] [color={rgb, 255:red, 0; green, 0; blue, 0 }  ][line width=0.75]    (10.93,-3.29) .. controls (6.95,-1.4) and (3.31,-0.3) .. (0,0) .. controls (3.31,0.3) and (6.95,1.4) .. (10.93,3.29)   ;
\draw  [fill={rgb, 255:red, 155; green, 155; blue, 155 }  ,fill opacity=1 ] (377,111) -- (447,111) -- (447,151) -- (377,151) -- cycle ;
\draw    (375.5,133) -- (312,132.03) ;
\draw [shift={(310,132)}, rotate = 360.87] [color={rgb, 255:red, 0; green, 0; blue, 0 }  ][line width=0.75]    (10.93,-3.29) .. controls (6.95,-1.4) and (3.31,-0.3) .. (0,0) .. controls (3.31,0.3) and (6.95,1.4) .. (10.93,3.29)   ;
\draw  [fill={rgb, 255:red, 241; green, 237; blue, 237 }  ,fill opacity=1 ] (426,44.56) -- (472.42,44.56) -- (472.42,75) -- (426,75) -- cycle ;
\draw  [fill={rgb, 255:red, 235; green, 233; blue, 233 }  ,fill opacity=1 ] (462.69,26.9) -- (509.11,26.9) -- (509.11,57.33) -- (462.69,57.33) -- cycle ;
\draw  [fill={rgb, 255:red, 230; green, 230; blue, 230 }  ,fill opacity=1 ] (501.36,12) -- (547.78,12) -- (547.78,42.44) -- (501.36,42.44) -- cycle ;
\draw  [color={rgb, 255:red, 208; green, 2; blue, 27 }  ,draw opacity=1 ][fill={rgb, 255:red, 208; green, 2; blue, 27 }  ,fill opacity=1 ] (431.23,68.85) .. controls (431.23,67.61) and (432.4,66.59) .. (433.85,66.59) .. controls (435.29,66.59) and (436.47,67.61) .. (436.47,68.85) .. controls (436.47,70.1) and (435.29,71.12) .. (433.85,71.12) .. controls (432.4,71.12) and (431.23,70.1) .. (431.23,68.85) -- cycle ;
\draw  [color={rgb, 255:red, 208; green, 2; blue, 27 }  ,draw opacity=1 ][fill={rgb, 255:red, 208; green, 2; blue, 27 }  ,fill opacity=1 ] (469.8,46.83) .. controls (469.8,45.58) and (470.98,44.56) .. (472.42,44.56) .. controls (473.87,44.56) and (475.04,45.58) .. (475.04,46.83) .. controls (475.04,48.07) and (473.87,49.09) .. (472.42,49.09) .. controls (470.98,49.09) and (469.8,48.07) .. (469.8,46.83) -- cycle ;
\draw  [fill={rgb, 255:red, 208; green, 2; blue, 27 }  ,fill opacity=1 ] (521.95,29.48) .. controls (521.95,28.23) and (523.12,27.22) .. (524.57,27.22) .. controls (526.01,27.22) and (527.18,28.23) .. (527.18,29.48) .. controls (527.18,30.73) and (526.01,31.74) .. (524.57,31.74) .. controls (523.12,31.74) and (521.95,30.73) .. (521.95,29.48) -- cycle ;
\draw  [color={rgb, 255:red, 193; green, 222; blue, 255 }  ,draw opacity=1 ][fill={rgb, 255:red, 193; green, 222; blue, 255 }  ,fill opacity=1 ] (480.24,39.91) -- (492.55,39.91) -- (492.55,50.55) -- (480.24,50.55) -- cycle ;
\draw  [color={rgb, 255:red, 193; green, 222; blue, 255 }  ,draw opacity=1 ][fill={rgb, 255:red, 193; green, 222; blue, 255 }  ,fill opacity=1 ] (442.07,58.64) -- (454.38,58.64) -- (454.38,69.28) -- (442.07,69.28) -- cycle ;
\draw  [color={rgb, 255:red, 193; green, 222; blue, 255 }  ,draw opacity=1 ][fill={rgb, 255:red, 193; green, 222; blue, 255 }  ,fill opacity=1 ] (518.41,24.16) -- (530.72,24.16) -- (530.72,34.8) -- (518.41,34.8) -- cycle ;
\draw  [color={rgb, 255:red, 208; green, 2; blue, 27 }  ,draw opacity=1 ][fill={rgb, 255:red, 208; green, 2; blue, 27 }  ,fill opacity=1 ] (521.95,29.48) .. controls (521.95,28.23) and (523.12,27.22) .. (524.57,27.22) .. controls (526.01,27.22) and (527.18,28.23) .. (527.18,29.48) .. controls (527.18,30.73) and (526.01,31.74) .. (524.57,31.74) .. controls (523.12,31.74) and (521.95,30.73) .. (521.95,29.48) -- cycle ;
\draw  [fill={rgb, 255:red, 241; green, 237; blue, 237 }  ,fill opacity=1 ] (392,227.56) -- (436.54,227.56) -- (436.54,258) -- (392,258) -- cycle ;
\draw  [fill={rgb, 255:red, 235; green, 233; blue, 233 }  ,fill opacity=1 ] (427.21,209.9) -- (471.75,209.9) -- (471.75,240.33) -- (427.21,240.33) -- cycle ;
\draw  [fill={rgb, 255:red, 230; green, 230; blue, 230 }  ,fill opacity=1 ] (464.31,195) -- (508.85,195) -- (508.85,225.44) -- (464.31,225.44) -- cycle ;
\draw  [color={rgb, 255:red, 208; green, 2; blue, 27 }  ,draw opacity=1 ][fill={rgb, 255:red, 208; green, 2; blue, 27 }  ,fill opacity=1 ] (397.02,251.85) .. controls (397.02,250.61) and (398.15,249.59) .. (399.53,249.59) .. controls (400.92,249.59) and (402.04,250.61) .. (402.04,251.85) .. controls (402.04,253.1) and (400.92,254.12) .. (399.53,254.12) .. controls (398.15,254.12) and (397.02,253.1) .. (397.02,251.85) -- cycle ;
\draw  [color={rgb, 255:red, 208; green, 2; blue, 27 }  ,draw opacity=1 ][fill={rgb, 255:red, 208; green, 2; blue, 27 }  ,fill opacity=1 ] (434.03,229.83) .. controls (434.03,228.58) and (435.15,227.56) .. (436.54,227.56) .. controls (437.93,227.56) and (439.05,228.58) .. (439.05,229.83) .. controls (439.05,231.07) and (437.93,232.09) .. (436.54,232.09) .. controls (435.15,232.09) and (434.03,231.07) .. (434.03,229.83) -- cycle ;
\draw  [fill={rgb, 255:red, 208; green, 2; blue, 27 }  ,fill opacity=1 ] (484.06,212.48) .. controls (484.06,211.23) and (485.19,210.22) .. (486.58,210.22) .. controls (487.96,210.22) and (489.09,211.23) .. (489.09,212.48) .. controls (489.09,213.73) and (487.96,214.74) .. (486.58,214.74) .. controls (485.19,214.74) and (484.06,213.73) .. (484.06,212.48) -- cycle ;
\draw  [color={rgb, 255:red, 193; green, 222; blue, 255 }  ,draw opacity=1 ][fill={rgb, 255:red, 193; green, 222; blue, 255 }  ,fill opacity=1 ] (444.04,222.91) -- (455.86,222.91) -- (455.86,233.55) -- (444.04,233.55) -- cycle ;
\draw  [color={rgb, 255:red, 193; green, 222; blue, 255 }  ,draw opacity=1 ][fill={rgb, 255:red, 193; green, 222; blue, 255 }  ,fill opacity=1 ] (407.42,241.64) -- (419.23,241.64) -- (419.23,252.28) -- (407.42,252.28) -- cycle ;
\draw  [color={rgb, 255:red, 193; green, 222; blue, 255 }  ,draw opacity=1 ][fill={rgb, 255:red, 193; green, 222; blue, 255 }  ,fill opacity=1 ] (480.67,207.16) -- (492.48,207.16) -- (492.48,217.8) -- (480.67,217.8) -- cycle ;
\draw  [color={rgb, 255:red, 208; green, 2; blue, 27 }  ,draw opacity=0.16 ][fill={rgb, 255:red, 208; green, 2; blue, 27 }  ,fill opacity=0.37 ] (484.06,212.48) .. controls (484.06,211.23) and (485.19,210.22) .. (486.58,210.22) .. controls (487.96,210.22) and (489.09,211.23) .. (489.09,212.48) .. controls (489.09,213.73) and (487.96,214.74) .. (486.58,214.74) .. controls (485.19,214.74) and (484.06,213.73) .. (484.06,212.48) -- cycle ;
\draw    (413,153) -- (320.2,209.95) ;
\draw [shift={(318.5,211)}, rotate = 328.46000000000004] [color={rgb, 255:red, 0; green, 0; blue, 0 }  ][line width=0.75]    (10.93,-3.29) .. controls (6.95,-1.4) and (3.31,-0.3) .. (0,0) .. controls (3.31,0.3) and (6.95,1.4) .. (10.93,3.29)   ;

\draw (259,43) node   [align=left] {Camera and Object Acquisition};
\draw (412,131) node   [align=left] {Noise};
\draw (294,132) node   [align=left] {+};
\draw (443.98,60.53) node   [align=left] {\textbf{A}};
\draw (481.66,40.95) node   [align=left] {\textbf{A}};
\draw (520.07,25.62) node   [align=left] {\textbf{A}};
\draw (409.25,243.53) node   [align=left] {\textbf{A}};
\draw (445.4,223.95) node   [align=left] {\textbf{A}};
\draw (480.67,207.16) node   [align=left] {\textbf{A}};
\draw (359,171) node    {$R_{p}$};
\draw (295,263) node   [align=left] {End User};

\end{resizedtikzpicture}
\caption{Privacy ensuring mechanism}
\label{fig:frames}
\end{figure}
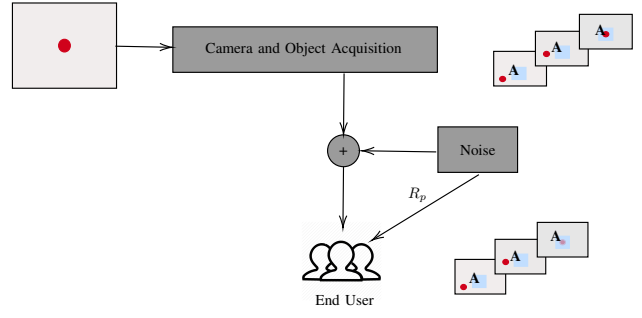
\section{Conclusion}
In this work we addressed two questions related to privacy and utility for moving object detection from a video stream using the Kalman filter. We modeled them as convex optimization problems based on LMIs. The proposed framework was implemented on a numerical problem for two scenarios. First, the purpose was to track an object with an upper bound on estimation error while ensuring utility. Second, we calculated the minimal noise that needs to be injected to a frame to ensure desired privacy prescribed by a lower bound on the localization error of the object.  

\section{Acknowledgment}
We are thankful to the reviewers whose valuable feedback helped in improving our work.

\bibliography{conference_101719}

\begin{thebibliography}{10}
\providecommand{\url}[1]{#1}
\csname url@samestyle\endcsname
\providecommand{\newblock}{\relax}
\providecommand{\bibinfo}[2]{#2}
\providecommand{\BIBentrySTDinterwordspacing}{\spaceskip=0pt\relax}
\providecommand{\BIBentryALTinterwordstretchfactor}{4}
\providecommand{\BIBentryALTinterwordspacing}{\spaceskip=\fontdimen2\font plus
\BIBentryALTinterwordstretchfactor\fontdimen3\font minus
  \fontdimen4\font\relax}
\providecommand{\BIBforeignlanguage}[2]{{%
\expandafter\ifx\csname l@#1\endcsname\relax
\typeout{** WARNING: IEEEtran.bst: No hyphenation pattern has been}%
\typeout{** loaded for the language `#1'. Using the pattern for}%
\typeout{** the default language instead.}%
\else
\language=\csname l@#1\endcsname
\fi
#2}}
\providecommand{\BIBdecl}{\relax}
\BIBdecl

\bibitem{Acquisti_2006}
A.~Acquisti and R.~Gross, ``Imagined communities: Awareness, information
  sharing, and privacy on the facebook,'' in \emph{Privacy Enhancing
  Technologies}.\hskip 1em plus 0.5em minus 0.4em\relax Springer Berlin
  Heidelberg, 2006, pp. 36--58.

\bibitem{kalman}
R.~E. Kalman, ``A new approach to linear filtering and prediction problems,''
  \emph{Transactions of the ASME--Journal of Basic Engineering}, vol.~82, no.
  Series D, pp. 35--45, 1960.

\bibitem{A__2010}
J.~A., F.~A., and F.~Torres, ``Kalman filtering for sensor fusion in a human
  tracking system,'' in \emph{Kalman Filter}.\hskip 1em plus 0.5em minus
  0.4em\relax {InTech}, may 2010.

\bibitem{Orekondy_2018}
T.~Orekondy, M.~Fritz, and B.~Schiele, ``Connecting pixels to privacy and
  utility: Automatic redaction of private information in images,'' in
  \emph{2018 {IEEE}/{CVF} Conference on Computer Vision and Pattern
  Recognition}.\hskip 1em plus 0.5em minus 0.4em\relax {IEEE}, jun 2018.

\bibitem{Boyle_2000}
M.~Boyle, C.~Edwards, and S.~Greenberg, ``The effects of filtered video on
  awareness and privacy,'' in \emph{Proceedings of the 2000 {ACM} conference on
  Computer supported cooperative work}.\hskip 1em plus 0.5em minus 0.4em\relax
  {ACM} Press, 2000.

\bibitem{Winkler_2011}
T.~Winkler and B.~Rinner, ``User-centric privacy awareness in video
  surveillance,'' \emph{Multimedia Systems}, vol.~18, no.~2, pp. 99--121, jul
  2011.

\bibitem{Qureshi_2009}
F.~Z. Qureshi, ``Object-video streams for preserving privacy in video
  surveillance,'' in \emph{2009 Sixth {IEEE} International Conference on
  Advanced Video and Signal Based Surveillance}.\hskip 1em plus 0.5em minus
  0.4em\relax {IEEE}, sep 2009.

\bibitem{Brassil_2009}
J.~Brassil, ``Technical challenges in location-aware video surveillance
  privacy,'' in \emph{Protecting Privacy in Video Surveillance}.\hskip 1em plus
  0.5em minus 0.4em\relax Springer London, 2009, pp. 91--113.

\bibitem{Saho_2018}
K.~Saho, ``Kalman filter for moving object tracking: Performance analysis and
  filter design,'' in \emph{Kalman Filters - Theory for Advanced
  Applications}.\hskip 1em plus 0.5em minus 0.4em\relax {InTech}, feb 2018.

\bibitem{Kim_2015}
G.-W. Kim, ``An implementation of object detection and tracking algorithm using
  a fusion method of {SURF} and kalman filter,'' \emph{The Journal of Korean
  Institute of Information Technology}, vol.~13, no.~2, p.~59, feb 2015.

\bibitem{anderson1979optimal}
B.~D. Anderson and J.~B. Moore, ``Optimal filtering,'' \emph{Englewood Cliffs},
  vol.~21, pp. 22--95, 1979.

\bibitem{MATLAB:2017}
MATLAB, \emph{version 9.3.0.713579 (R2017b)}.\hskip 1em plus 0.5em minus
  0.4em\relax Natick, Massachusetts: The MathWorks Inc., 2017.

\bibitem{cvx}
M.~Grant and S.~Boyd, ``{CVX}: Matlab software for disciplined convex
  programming, version 2.1,'' Mar. 2014.

\bibitem{gb08}
------, ``Graph implementations for nonsmooth convex programs,'' in
  \emph{Recent Advances in Learning and Control}, ser. Lecture Notes in Control
  and Information Sciences, V.~Blondel, S.~Boyd, and H.~Kimura, Eds.\hskip 1em
  plus 0.5em minus 0.4em\relax Springer-Verlag Limited, 2008, pp. 95--110.

\bibitem{T_t_nc__2003}
R.~H. Tutuncu, K.~C. Toh, and M.~J. Todd, ``Solving
  semidefinite-quadratic-linear programs using {SDPT}3,'' \emph{Mathematical
  Programming}, vol.~95, no.~2, pp. 189--217, feb 2003.

\end{thebibliography}
\bibliographystyle{IEEEtran}

\end{document}